\documentclass[10pt]{article}
\usepackage{amsbsy}
\usepackage{amsmath,amsfonts,amssymb,amsthm,mathrsfs,mathtools,stmaryrd}

\usepackage{doi,url}
\usepackage[numbers]{natbib}
\usepackage{etoolbox,letltxmacro,xifthen,ifdraft} 
\usepackage[dvipsnames]{xcolor}
\usepackage{amsmath,amssymb,enumitem,bbm,mathrsfs,amsthm,aliascnt,mathtools}
\usepackage{graphicx,subcaption}
\usepackage[retainorgcmds]{IEEEtrantools}
\usepackage{hyperref} 
\hypersetup{final,hidelinks}
\usepackage[atend]{bookmark} 

\providecommand{\email}[1]{\href{mailto:#1}{\nolinkurl{#1}}}

\setlist[enumerate,1]{label={(\roman*)}}
\setlist[enumerate,2]{label={(\alph*)}}
\setlist[enumerate,3]{label={(\Roman*)}}

\newcommand{\newsstheorem}[2]{
  \newaliascnt{#1}{dummy}
  \newtheorem{#1}[#1]{#2}
  \aliascntresetthe{#1}
  \expandafter\def\csname #1autorefname\endcsname{#2}
}
\numberwithin{dummy}{section}
\theoremstyle{plain}
  \newsstheorem{theorem}{Theorem}
  \newsstheorem{proposition}{Proposition}
  \newsstheorem{corollary}{Corollary}
  \newsstheorem{lemma}{Lemma}
  \newsstheorem{definition}{Definition}
  \theoremstyle{definition}
  \newsstheorem{example}{Example}
  \newsstheorem{assumption}{Assumption}
\theoremstyle{remark}
  \newsstheorem{remark}{Remark}

\newenvironment{eqnarr*}{\begin{IEEEeqnarray*}{rCl}}{\end{IEEEeqnarray*}\ignorespacesafterend}

\newcommand\RR{\mathbb{R}}

\newcommand\PP{\mathbb{P}}

\newcommand\EE{\mathbb{E}}
\newcommand\NN{\mathbb{N}}
\newcommand\ZZ{\mathbb{Z}}

\newcommand\e{{\mathrm e}}

\makeatletter \newcommand\mathof[1]{{\operator@font#1}} \makeatother

\newcommand\dd{\mathof{d}}

\begin{document}

\title{Universality of \\Noise Reinforced Brownian Motions}
\author{Jean Bertoin\footnote{Institute of Mathematics, University of Zurich, Switzerland, \texttt{jean.bertoin@math.uzh.ch}}  }
\date{\small Dedicated to the memory of Vladas Sidoravicius}
\maketitle 
\thispagestyle{empty}

\begin{abstract} 
A noise reinforced Brownian motion  is a centered Gaussian process  $\hat B=(\hat B(t))_{t\geq 0}$ with covariance 
$$\EE(\hat B(t)\hat B(s))=(1-2p)^{-1}t^ps^{1-p} \quad \text{for} \quad 0\leq s \leq t,$$
 where $p\in(0,1/2)$ is a reinforcement parameter.  Our main purpose is to establish a version of Donsker's invariance principle for a large family of step-reinforced random walks in the diffusive regime, and more specifically, to show that  $\hat B$ arises as the universal scaling limit of the former. This extends known results on the asymptotic behavior of the so-called elephant random walk. 
\newline  \vskip 1mm
{\normalfont \bfseries Keywords:}
Reinforcement, Brownian motion, Invariance Principle, Elephant Random Walk.\newline
\vskip 1mm
{\normalfont \bfseries Mathematics Subject Classification:}  60G50; 60G51;  60K35.

\end{abstract}

\section{Introduction}\label{s:intro}
This work concerns a rather simple real-valued and centered Gaussian process $\hat B=(\hat B(t))_{t\geq 0}$ with covariance function
\begin{equation}\label{E:cov}
\EE(\hat B(t)\hat B(s))=\frac{t^ps^{1-p}}{1-2p}\qquad \text{ for }0\leq s \leq t,
\end{equation}  
where $p\in(0,1/2)$ is a fixed parameter. Recently, this process has notably appeared  as the scaling limit for diffusive regimes of the so-called elephant random walk, a  simple random walk with memory that has been introduced  by Sch\"utz and Trimper~\cite{SchTr}. Just as the standard Brownian motion $B$ corresponds to the integral of a white noise, 
 $\hat B$ can be thought of as the integral of a reinforced version of the white noise, hence the name
 \textit{noise reinforced Brownian motion}. Here, reinforcement means  that the noise tends to repeat itself infinitesimally as time passes; we refer to \cite{Pem} for a survey of various models of stochastic processes with reinforcement and their applications, and to \cite{KV} and works cited therein for more recent contributions in this area.  The parameter $p$ should be interpreted  as the strength of the reinforcement; specifically, it represents the probability that an infinitesimal portion of the noise is a repetition. In the limiting case $p=0$ without reinforcement, one just recovers the standard Brownian motion.
 
 Our purpose here is twofold. We will first present several basic properties of $\hat B$ that mirror well-known facts for the standard Brownian motion, even though the laws of $B$ and $\hat B$ are mutually singular.  We will then establish a version of Donsker's invariance principle. That is, we will show that  any so-called step-reinforced random walk with reinforcement parameter $p$, whose typical step has a finite second moment, converges  after the usual centering and rescaling to a noise reinforced Brownian motion.

This invariance principle for step-reinforced random walks  has been established previously 
for elephant random walks, that is in the special case when the typical step has the Rademacher law; see \cite{BaurBer, ColGavSch2}.
Technically,  \cite{ColGavSch2} uses Skorokhod's embedding and relies crucially on the assumption that the increments take values in $\{-1,+1\}$, whereas 
  \cite{BaurBer} uses  limit theorems for generalized P\'olia urns due to  Janson \cite{Janson}. The latter approach is not available for arbitrary step distributions either, notably  as one would need to work with urn models having types in an infinite space (the real line $\RR$ to be more specific), and how to deal with this kind of urns is still an open problem; see Remark 4.1 in \cite{Janson}.

We shall therefore  follow here a different method and rather embed a  step-reinforced random walk in a branching process with types in $\RR$, using a time-substitution via an independent Yule process. This yields a remarkable martingale, whose quadratic variation can then be estimated from well-known properties of Yule processes. In turn, this enables the application of the martingale functional central limit theorem (later on referred to as martingale FCLT).

\section{Some basic properties}
We start by observing from \eqref{E:cov} that the noise reinforced Brownian motion admits a simple representation as a Wiener  integral, namely
\begin{equation}\label{E:Wrep}
\hat B(t)= t^p\int_0^t s^{-p} \dd B(s), \qquad t\geq 0,
\end{equation}
where $B=(B(s))_{s\geq 0}$ is a standard Brownian motion. Note that equivalently, 
\begin{equation}\label{E:Wrep2}
 \hat B \text{ has the same law as  }\left(  \frac{ t^p}{\sqrt{1-2p}} B(t^{1-2p})\right)_{t\geq 0}.
\end{equation}

It follows immediately from \eqref{E:Wrep2} and the classical law of the iterated logarithm for $B$ (see, e.g. Theorem II.1.19 in \cite{RY}) that
\begin{equation} \label{E:LIL} \limsup_{t\to \infty} \frac{\hat B(t)}{\sqrt{2t \ln \ln t}} = \limsup_{t\to 0+} \frac{\hat B(t)}{\sqrt{2t \ln \ln (1/t)}} = \frac{1}{\sqrt{1-2p}} \qquad \text{a.s.}
\end{equation}
  In particular, we see  that for different reinforcement parameters $p$, the distributions of noise reinforced Brownian motions, say on the time interval $[0,1]$, yield laws on ${\mathcal C}([0,1])$ which are mutually singular, and are also singular with respect to the Wiener measure.

We deduce from \eqref{E:Wrep} by stochastic calculus that $\hat B$ is a semi-martingale which solves the stochastic differential equation
\begin{equation} \label{E:eds}
 \dd \hat B(t)= \dd B(t) + \frac{p}{t} \hat B(t) \dd t,\qquad \hat B(0)=0.
 \end{equation}
This shows that $\hat B$  is actually a time-inhomogeneous diffusion process with quadratic variation $\langle \hat B\rangle(t)=t$. 
In this direction, we also infer from \eqref{E:cov} that the process 
$$
 \hat b(t)= \hat B(t)- t^{1-p} \hat B(1) \qquad \text{for }0\leq t \leq 1,
$$
 is independent of  $\hat B(1)$. Hence, for any $x\in\RR$, 
 $$\hat b(t)+t^{1-p} x=\hat B(t)+t^{1-p}(x-\hat B(1))\qquad \text{for }0\leq t \leq 1,$$ is  a version of the bridge of $\hat B$ from $0$ to $x$ with unit duration. These properties should be viewed as the reinforced versions of the classical construction of
  Brownian bridges; see for instance  \cite{RY} on page 37.

We further see from \eqref{E:cov} that the scaling property,
\begin{equation}\label{E:scal}
\text{for every $c>0$,  }\left( c^{-1} \hat B(c^2 t)\right)_{t\geq 0} \text{ has the same law as }\hat B,
 \end{equation}
as well as the time-inversion property, 
$$
 \left( t \hat B(1/t)\right)_{t> 0} \text{ has the same law as } \left( \hat B(t)\right)_{t> 0},
$$
both hold. Needless to say, these two properties are also fulfilled by the standard Brownian motion;
see, e.g. Proposition I.1.10 in \cite{RY}. In this vein, we also point at a remarkable connection with Ornstein-Uhlenbeck processes:  
 the process
 $$\hat U(t)=\e^{-t/2} \hat B(\e^t)\,, \qquad t\in \RR$$
 is a stationary Ornstein-Uhlenbeck process with infinitesimal generator 
 $${\mathcal G}f(x)= \frac{1}{2} f''(x)+(p-1/2) x f'(x),$$ where $f\in{\mathcal C}^2(\RR)$. 
  This can be checked by stochastic calculus from \eqref{E:eds}, or directly by observing  from \eqref{E:cov} and  \eqref{E:scal} that $\hat U$ is a stationary Gaussian 
 process with covariance 
 $$\EE(\hat U(t)\hat U(0))= \frac{\e^{(p-1/2)t}}{1-2p} \qquad \text{for }t\geq 0.$$

 On the other hand, several classical results for the standard Brownian motion plainly  fail for its reinforced  version. For instance, the increments of $\hat B$ are clearly not independent, and the time-reversal property (e.g. Exercise 1.11 in Chapter I of \cite{RY}) also fails.

\section{An invariance principle with reinforcement}

The main purpose of this section is to point out that noise reinforced Brownian motions arise as universal scaling limits of a large class of random walks with step reinforcement. We first recall some features on the so-called elephant random walk, where  the story begins.

The \textit{elephant random walk} has been introduced by Sch\"utz and Trimper~\cite{SchTr} 
as a discrete-time nearest neighbor process with memory on $\ZZ$; it can be depicted as follows. 
Fix some $q\in(0,1)$ and call $q$ the memory parameter. 
Imagine that a walker (an elephant) makes a first step in $\{-1,+1\}$ at time $1$; then
at each time $n\geq 2$,  it selects
randomly a step from its past. With probability $q$, the elephant
repeats this step, and with complementary probability $1-q$, it makes the opposite step.  Note that for $q=1/2$, the elephant merely follows the path of a simple symmetric random walk. The elephant random walk has generated much interest in the recent years, we refer notably to \cite{BaurBer, Bercu, ColGavSch1, ColGavSch2, Copapa, KuTa, Kur}, see also \cite{Baur, Bercu2, NRLP, Buart, GutStadt} for variations, and references therein for further related works. A remarkable feature  is  that the large time asymptotic behavior of an elephant random walk
is diffusive when the memory parameter $q$ is less than $3/4$  and super-diffusive
when $q>3/4$.
  \begin{remark} The laws of the iterated logarithm \eqref{E:LIL} for a noise reinforced Brownian motion bear the same relation to that for the elephant random walk in the diffusive regime (Corollary 1 in \cite{ColGavSch2} and Theorem 3.2 in \cite{Bercu}; beware that the memory parameter denoted by $p$ there corresponds to $q=(1+p)/2$ in the present notation as it will be stressed below), as the law of the iterated logarithm for the standard Brownian motion due to P. L\'evy does to that for the simple random walk due to A.Y. Khintchine.
  \end{remark}

We are interested in a generalization where the distribution of a typical step of the walk is arbitrary, that we call \textit{step reinforced random walk}. Fix a parameter $p\in(0,1)$, called the \textit{reinforcement parameter}\footnote{Beware that we assumed $p<1/2$ in the preceding sections. The first part  of the present section (super-diffusive regime) does not involve any noise reinforced Brownian motion, and the case $p\geq 1/2$ is allowed. In the second part of this section (diffusive regime), we shall again focus on the case $p<1/2$ and noise reinforced Brownian motions will then re-appear.}; at each discrete time, with probability  $p$, a step reinforced random walk repeats one of its preceding steps chosen uniformly at random, and otherwise, i.e. with probability $1-p$, it has an independent increment with a fixed distribution.
More precisely, consider  a sequence $X_1, X_2, \ldots$ of i.i.d. copies of a random variable $X$ in $\RR$ and define recursively $\hat X_1, \hat X_2, \ldots$ as follows. Let  $(\varepsilon_i: i\geq 2)$ be an independent sequence of Bernoulli variables with parameter $p$.  
We set first $\hat X_1=X_1$, and next for $i\geq 2$, we let $\hat X_i=X_i$ if $\varepsilon_i=0$, whereas   we define $\hat X_i$ as a uniform random sample from $\hat X_1, \ldots, \hat X_{i-1}$   if $\varepsilon_i=1$.
Finally, the sequence of the partial sums
$$\hat S(n)=\hat X_1+ \cdots + \hat X_n, \qquad n\in\NN,$$ 
is referred to as  a step reinforced random walk. We stress that in general, the Markov property fails for $\hat S$ (even though it may hold for certain specific step distributions).

 When the typical step $X$ has the Rademacher law, i.e. $\PP(X=1)=\PP(X=-1)=1/2$,   K\"ursten \cite{Kur} (see also \cite{G-NL}) pointed out that $\hat S$ is  a version of the  elephant random walk  with memory parameter $q=(p+1)/2$ in the present notation.  
  When $X$ has a symmetric stable distribution, $\hat S$ is the so-called shark random swim which has been studied in depth by
Businger \cite{Buart}. More general versions when the distribution of $X$ is infinitely divisible have been considered in \cite{NRLP}. 

The main result of Businger \cite{Buart} is that the large time asymptotic behavior of a shark random swim exhibits 
 a phase transition similar to that for the elephant random walk, but for a different critical parameter. We shall now extend this to a large class of step reinforced random walks.
In the sequel we implicitly rule out the degenerate case when  the typical step variable  $X$ is  a constant. 
We start with the super-diffusive regime for which  all the ingredients are already in Section 3.1 in \cite{Buart}.

\begin{theorem}\label{T1} 
 Let $p\in(1/2,1)$, and suppose that $X\in L^2(\PP)$. Then 
$$\lim_{n\to \infty} \frac{\hat S(n)-n\EE(X)}{n^p}= L\qquad\text{in } L^2(\PP),$$
where $L$ is some non-degenerate random variable.
\end{theorem}
\begin{proof} By centering and normalizing, we  may assume without loss of generality that $\EE(X)=0$ and $\textrm{Var}(X)=1$. One first introduces  for every $i,n\in\NN$, the number $r_{i,n}$ of repetitions of the variable $X_i$ in the reinforced sequence $\hat X_1, \ldots, \hat X_n$ (in particular, $r_{i,n}=0$ when either $\varepsilon_i=1$ or $i>n$). We stress that repetition numbers only depend on the Bernoulli variables $\varepsilon_j$ and on the uniform random samples among the
previous steps, and 
are hence independent of the variables $X_j$. 
We know from Lemmas 3 and 5 in \cite{Buart} that for each $i\in\NN$
 $$\lim_{n\to \infty} n^{-p} r_{i,n}=R_i, \qquad \text{a.s.}$$
where $R_i$ is some non-degenerated random variable with $\sum_i \EE\left(R_i^2\right)<\infty$. 
 
 In particular, 
$\sum_i  R_i^2<\infty$ a.s., and since the variables $X_i$ are i.i.d. centered and with unit variance, the sum $\sum_i R_i X_i=L$ is well-defined a.s. (as a martingale limit, conditionally on the $R_i$). More precisely, $\EE(L)=0$ and $\EE(L^2)=\sum_i \EE\left(R_i^2\right)$.

To conclude, we recall from Equation (6) in \cite{Buart}  that
 \begin{equation} \label{E:bu6}
 \lim_{n\to \infty} \sum_i \EE\left( \left| n^{-p} r_{i,n} -R_i
 \right|^2\right) =0.
 \end{equation}
By the construction of the step reinforced random walk, we have
$$\hat S(n)=\sum_{i} r_{i,n} X_i.$$
Since the variables $r_{i,n}$ and $R_i$ are independent of the $X_i$, and further  $\EE(X)=0$ and $\EE(X^2)=1$, there is the identity 
$$ \EE\left( | n^{-p} \hat S(n)-L|^2\right) = 
\EE\left( \sum_i \left| n^{-p} r_{i,n} -R_i
 \right|^2\right).
$$
An appeal to \eqref{E:bu6} completes the proof.
\end{proof}

We then turn our attention to the diffusive regime, and obtain a version of Donsker's invariance principle for step reinforced random walks. In this direction, we refer to Chapter VI in \cite{JS} for background on the Skorokhod topology and weak convergence of stochastic processes. 

\begin{theorem}\label{T2} 
 Let $p\in(0,1/2)$, and suppose that  $X \in  L^2(\PP)$. Then as $n\to \infty$,  the sequence of processes
$$\frac{ \hat S(\lfloor tn\rfloor) -tn\EE(X)}{\sqrt { n \mathrm{Var}(X)}} \,, \qquad t\geq 0$$
converges in distribution in the sense of Skorohod  
towards a noise reinforced Brownian motion $\hat B$ with reinforcement parameter $p$. 
\end{theorem}

Theorem \ref{T2} will be established in the next section, first under the additional  assumption that the typical step is bounded, and then in the general situation. 
 Our approach relies essentially on  the martingale FCLT, which also plays a key role for 
urn models (see \cite{Gouet, Janson}), as well in the works of Coletti \textit{et al.} \cite{ColGavSch1, ColGavSch2} and of Bercu \cite{Bercu} on the elephant random walk. 
For the sake of simplicity, we consider here the one-dimensional setting only, however the argument could be readily adapted to
$\RR^d$; see also \cite{Bercu2}.

We first embed the step reinforced random walk in a branching process as follows. 
Let $Y=(Y_t)_{t\geq 0}$ be a standard Yule process, that is $Y$ is a pure birth process started from $Y_0=1$, with birth rate $n$ from any state $n\in\NN$. We further assume  that $Y$ and $\hat S$ are independent, and will mainly work with the time-changed process $\hat S(Y_{\cdot})$. It may be worth dwelling a bit on our motivation for introducing this time-substitution; even though this discussion will not be used in the proof, it nonetheless provides a useful guiding line for our approach. Roughly speaking, $Y$ describes   a population model started from a single ancestor,
where  individuals are eternal, and each begets a child at unit rate, independently of the other individuals. 
Those individuals are naturally enumerated in the increasing order of their birth time, in particular the ancestor is the first individual. 
We decide to assign the type $\hat X_n$ to the $n$-th individual, and write
 $$\mathbf{Z}_t(\dd x)= \sum_{n=1}^{Y_t} \delta_{\hat X(n)}(\dd x), \qquad x\in \RR,$$
 for the point process of the types of individuals alive at time $t$. 
 We see from basic properties of  independent exponential clocks that $\mathbf{Z}=(\mathbf{Z}_t)_{t\geq 0}$ is a (multitype) branching process,
 in which each individual begets a child at unit rate, independently of the other individuals. A child is either a clone of its parent, an event which occurs with probability $p$, or a mutant, an event which occurs with probability $1-p$. If a child is a clone, then its type is the same as that of its parent, whereas if it is a mutant, its type is given by an independent copy of $X$. 
In this setting, there are the identities
 \begin{equation}\label{E:SZ}
 Y_t= \mathbf{Z}_t({\mathbf 1}) \quad \text{and} \quad 
 \hat S(Y_t) = \mathbf{Z}_t (\mathrm {Id}),\end{equation}
 with the notation  ${\mathbf 1}(x)=1$, $\mathrm{Id}(x)=x$ and  $ \mathbf{Z}_t(f)= \int_{\RR} f(x) \mathbf{Z}_t(\dd x)$.  The function  $\mathbf{z}  \mapsto \mathbf{z}(\mathbf{1})$  on the space of point measures  $\mathbf{z}$
 is always an eigenfunction for the infinitesimal generator of the branching process $\mathbf{Z}$ for the eigenvalue $1$. The key point is that when the typical step is centered, $\EE(X)=0$, the function $\mathbf{z}  \mapsto \mathbf{z}(\mathrm{Id})$ is also an eigenfunction, now for the eigenvalue $p$. 
 
  \section{Proof of the invariance principle}
 
This section is devoted to the proof of Theorem \ref{T2}. Without loss of generality, we henceforth assume that $X\in L^2(\PP)$ with $\EE(X)=0$, and set $\sigma^2=\EE(X^2)$. 
 We shall first give some preliminary estimates related to a remarkable martingale, 
 then we shall establish Theorem \ref{T2} under the additional assumption that the variable $X$ is bounded, and finally we shall show how this constraint can be removed.

\subsection{On a remarquable martingale}
Recall that we implicitly assume that $X$ is centered. 
\begin{lemma} \label{L2} The process
 $$ M(t)= \e^{-pt} \hat S(Y_t),\quad \text{for }t\geq 0,$$
  is a square integrable  martingale with finite variation.
  \end{lemma}
\begin{proof} Obviously, the set of jump times of $M$ is discrete. Since $M$ decays continuously between two consecutive jump times, \textit{a fortiori}  its paths 
have finite variation. 
Plainly, $\EE(\hat S(n)^2) \leq n^2 \EE(X^2)=n^2\sigma^2$, and since $\hat S$ and  $Y$ are independent with $\EE(Y_t^2)<\infty$, $M(t)$ is indeed square integrable.  

We next point out that for any $n\in \NN$,
\begin{align*}
\EE(\hat X_{n+1}\mid \hat X_1, \ldots, \hat X_n) &= (1-p)\EE(X) + p \frac{\hat X_1+ \cdots +\hat X_n}{n}\\
&=p\frac{\hat S(n)}{n}
\end{align*}
  (this observation is also the starting point of the analysis of the elephant random walk in \cite{ColGavSch1, ColGavSch2, Bercu}). 
  Since the Yule process has precisely jump rate $n$ from the state $n$, this entails that
     $$ \hat S(Y_t) -p\int_0^t  \hat S(Y_s) \dd s, \quad \text{for }t\geq 0$$
 is a martingale, and our statement then follows from elementary stochastic calculus. 
 \end{proof}
 
 We next derive numerical bounds for the second moment of the supremum process of the martingale $M$.  
 
 \begin{lemma} \label{L3} For every $t\geq 0$, one has
 $$\EE\left( \sup_{0\leq s \leq t} M^2(s) \right)  \leq \frac{4\sigma^2}{1-2p} \e^{(1-2p)t}.$$
 \end{lemma}
 \begin{proof} Since $M$ has finite variation, its square-bracket process  can be expressed in the form
\begin{equation}\label{E:sqrbrack} [M](t)=\int_{(0,t]} \e^{-2p s} |\hat X_{Y_s}|^2 \dd Y_s;
\end{equation}
see for instance Theorem 26.6(viii) in  \cite{Kal}. 
Recall that the instantaneous jump rate of the Yule process at time $s$  equals $Y_s$ and 
that the sequence $\hat X_1, \hat X_2, \ldots$ is independent of the Yule process with $\EE(\hat X_j^2)=\sigma^2$ for every $j\geq 1$. 
It follows that
$$ \EE([M](t))=\sigma^2 \int_0^t \e^{-2p s} \EE(Y_s) \dd s=  \frac{\sigma^2}{1-2p} \left(\e^{(1-2p)t}-1\right),$$
where for the second equality, we used  $\EE(Y_s)=\e^{s}$. 
This yields our claim by an appeal to the Burkholder-Davis-Gundy inequality; see, e.g. Theorem 26.12 in \cite{Kal}.  \end{proof} 

We then obtain bounds for the second moments of the supremum process of the step reinforced random walk itself, which will be useful later one. 
\begin{corollary}\label{C1} For every $n\geq 2$, we have
$$\EE \left (\max_{k\leq n}|\hat S(k)|^2 \right)  \leq \frac{4 \e^a \sigma^2 }{1-2p} n,$$
with $a=-\inf_{0<x\leq 1/2} x^{-1}\ln(1-x)$.
\end{corollary}
\begin{proof} Since $\hat S$ and $Y$ are independent and $Y$ is a counting process, we have for every $n\geq 1$ and $t\geq 0$ that 
$$\EE(\max_{k\leq n}|\hat S(k)|^2) \PP(Y_t> n) \leq \EE(\sup_{s\leq t} |\hat S(Y_s)|^2) \leq  \e^{2pt} \EE\left( \sup_{0\leq s \leq t} M^2(s) \right).$$
Take $t=\ln n$ and recall that $Y_t$ hat the geometric distribution with parameter $\e^{-t}=1/n$. So 
$$ \PP(Y_t> n)=(1-1/n)^n\geq \e^{-a} \qquad \text{for all }n\geq 2,$$  and we conclude the proof using Lemma \ref{L3}. 
\end{proof}

 \subsection{Proof of Theorem \ref{T2} when $X$ is bounded}
In this section, we shall prove Theorem \ref{T2} under the additional assumption that the typical step $X$ is a bounded variable. 
 We first estimate the angle-bracket $\langle M\rangle $ of $M$, and in this direction, we introduce 
 $$\hat V(n)= \hat X(1)^2+\cdots + \hat X(n)^2\,, \qquad n\in\NN.$$
 \begin{lemma}\label{L-1} Assume that $\|X\|_{\infty}<\infty$.
 We have
 $$ \EE\left( \left |\e^{-pt} \hat V(Y_t)- (1-p) \sigma^2 \int_0^t \e^{-ps} Y_s \dd s
 \right|^2\right) =o(\e^{2(1-p)t})\qquad \text{as } t\to \infty.$$
   \end{lemma} 
  \begin{proof}
Just as in the proof of Lemma \ref{L2}, we note the identity
$$\EE(\hat X(n+1)^2\mid \hat X(1), \ldots, \hat X(n)) =p\frac{\hat V(n)}{n}+(1-p)\sigma^2,$$
and get by stochastic calculus that the process
$$M'(t)=\e^{-pt} \hat V(Y_t)- (1-p)\sigma^2 \int_0^t \e^{-ps} Y_s \dd s, \qquad t\geq 0$$
is a square integrable martingale with finite variation and square bracket $$[M'](t)=\int_{(0,t]} \e^{-2ps} |\hat X(Y_s)|^4 \dd Y_s, \qquad t\geq 0.$$
We compute its expected value and get
\begin{align*} \EE(X^4) \EE\left( \int_{(0,t]} \e^{-2ps}  \dd Y_s\right) 
&= \EE(X^4)  \int_0^t\e^{(1-2p)s}  \dd s \\
&= o(\e^{2(1-p)t}),
\end{align*}
where for the first equality, we used that the instantaneous jump rate of the Yule process at time $s$  equals $Y_s$ and $\EE(Y_s)=\e^s$. 
  \end{proof}

 We next recall that the asymptotic behavior of the Yule process is described by
\begin{equation}\label{E:KS}
\lim_{t\to \infty} \e^{-t} Y_t= \tau\qquad \text{a.s.}
\end{equation} 
where $\tau$ is a standard exponential variable. Needless to say, $\tau$  is independent of $\hat S$.

 \begin{corollary} \label{C3}   Assume that $\|X\|_{\infty}<\infty$. The angle-bracket 
 $\langle M\rangle $ of the martingale $M$ in Lemma \ref{L2} fulfills
 $$\lim_{t\to \infty} \e^{-(1-2p)t} \langle M \rangle(t)= \frac{\tau}{1-2p} \qquad \text{in probability.} $$
 \end{corollary}
 
 \begin{proof} Recall that the square-bracket process of $M$  is given by \eqref{E:sqrbrack}; if follows readily that its predictable compensator is
  \begin{equation}\label{E:angleb}
\langle M\rangle(t)=\int_{(0,t]} \e^{-2p s}\left( p\hat V(Y_s)+(1-p) Y_s\right) \dd s, \qquad t\geq 0.
\end{equation}
 Since $2p<1$, our statement now derives from \eqref{E:KS}, \eqref{E:angleb} and Lemma \ref{L-1}. \end{proof}

The angle-bracket $\langle M \rangle$ is a continuous strictly increasing bijection from $\RR_+$ to $\RR_+$ a.s., see \eqref{E:angleb}, 
and we write $T$ for the inverse bijection. 
We  introduce for each $n\in \NN$ the process
$$N_n(t)=n^{-1/2} M_{T(nt)} = n^{-1/2} \e^{-pT(nt)} \hat S(Y_{T(nt)}), \qquad t\geq 0.$$
We are in position of applying the martingale FCLT.
\begin{proposition} \label{P1} Assume that $\|X\|_{\infty}<\infty$. As $n\to \infty$,  the sequence of processes $N_n$ converges in distribution in the sense of Skorohod
to a standard Brownian motion.
\end{proposition}
\begin{proof} Plainly, each $N_n$ is a square-integrable martingale with angle-bracket  $\langle N_n \rangle(t)=t$, and obviously its
maximum jump  is asymptotically negligible in $L^2(\PP)$ since $X$ is bounded.
This enables us to apply the martingale FCLT; see e.g. Theorem 2.1 in \cite{Whitt}, and also Section VIII.3 in \cite{JS} for more general versions.
\end{proof}
 
 We shall actually need a slightly stronger version of Proposition \ref{P1} in which the convergence holds conditionally on the variable $\tau$ in \eqref{E:KS}. 
 \begin{corollary}\label{C2} Assume that $\|X\|_{\infty}<\infty$. As $n\to \infty$, the sequence of pairs $(\tau,N_n)$ converges in distribution in the sense of Skorohod 
to $(\tau, B)$ where $B$ is a standard Brownian motion  independent of $\tau$.
  \end{corollary}
  
 \begin{proof} First fix $t>0$ and consider a random variable $A(t)$ which is measurable with respect to the sigma-algebra  $\sigma(M_{T(s)}: 0\leq s \leq t)$. 
 The statement with $A(t)$ replacing $\tau$ follows from the martingale FCTL applied to the sequence of processes $(N_n(s+t/n))_{s\geq 0}$, just  as in Proposition \ref{P1}. Corollary \ref{C2} can then be deduced,  provided that we can choose $A(t)$ such that 
 \begin{equation} \label{E:approx}
 \lim_{t\to \infty} A(t)=\tau\quad \text{a.s.}
 \end{equation}
Specifically, consider any continuous functional $F$ on the Skorohod space with $|F|\leq 1$, and any continuous and bounded function $f$ on $\RR$. 
Taking \eqref{E:approx} for granted, for any $\eta>0$ arbitrarily small, we can first choose $t>0$ sufficiently large so that 
$$\EE(|f(A(t))-f(\tau)|)\leq \eta,$$ 
and then $n(t)\in\NN$ such that
$$\left| \EE\left(f(A(t))F(N_n)\right) - \EE(f(A(t)))\EE(F(B)))\right| \leq \eta\qquad \text{for all } n\geq n(t).$$
We can conclude from the triangle inequality that
$$\left| \EE\left(f(\tau)F(N_n)\right) - \EE(f(\tau))\EE(F(B)))\right| \leq 3\eta\qquad \text{for all } n\geq n(t).$$

 We now have to construct variables $A(t)$ so that \eqref{E:approx} holds, and in this direction, we first assume that $\PP(X=0)=0$. All the steps of $\hat  S$ are non-zero a.s., and $Y_{T(t)}$ is then the total number of jumps of the process $M_{T(\cdot)}$ on the time interval $[0,t]$ (recall that this takes the initial jump at time $0$ into account), and hence  is measurable with respect to $\sigma(M_{T(s)}: 0\leq s \leq t)$. On the other hand, we deduce from Corollary \ref{C3} that
 \begin{equation} \label{E:AsT}
 T(t)= \frac{1}{1-2p}\ln\left(\frac{(1-2p)t}{\tau}\right) +o(1)\qquad \text{ as } t\to \infty,
 \end{equation}
 and then from \eqref{E:KS} that 
 \begin{equation} \label{E:AsY} 
 Y_{T(t)}\sim ((1-2p)t)^{1/(1-2p)} \tau^{-2p/(1-2p)}\qquad \text{ as } t\to \infty.
  \end{equation}
 The construction of the sought variables $A(t)$ is plain from \eqref{E:AsY}. 
 
 The case when $\PP(X=0)=a\in(0,1)$ only requires a minor modification. We can no longer identify $Y_{T(t)}$ with the total number of jumps of the process $M_{T(\cdot)}$ on the time interval $[0,t]$. Nonetheless the latter is now close to $(1-a)Y_{T(t)}$ for $t\gg 1$, and we can conclude just as above. 
 \end{proof} 
 
 We now have all the ingredients needed for the proof of the invariance principle.
 \begin{proof}[Proof of Theorem \ref{T2} when $X$ is bounded]
  
 Set $$s=s(t)=(\tau^{-2p}t)^{1/(1-2p)}\quad \text{and} \quad k=k(n)=((1-2p)n)^{1/(1-2p)},$$
 so \eqref{E:AsY}  and \eqref{E:AsT} yield respectively
 $$Y_{T(nt)} \sim ks \quad \text{and} \quad   \e^{-pT(nt)} \sim (ks/\tau)^{-p}.$$
 Since $ n = k^{1-2p}/(1-2p)$ and $t=\tau^{2p}s^{1-2p}$, we deduce from Corollary \ref{C2} and the very definition of the Skorohod topology involving changes of time (see, e.g. Section VI.1 in \cite{JS}) that as $k\to \infty$, the sequence of processes  $k^{-1/2}(\hat S(\lfloor ks\rfloor))_{s\geq 0}$ converges in distribution in the sense of Skorohod towards 
 $$\left( \frac{\tau^{-p}}{\sqrt {1-2p}} s^p B(\tau^{2p} s^{1-2p})\right)_{s\geq 0},$$
 where $B$ is a standard Brownian motion independent of $\tau$. By the scaling property of Brownian motion, the process displayed above has the same distribution as
 $$\left( \frac{s^p}{\sqrt {1-2p}} B(s^{1-2p})\right)_{s\geq 0},$$
 and we complete the proof with \eqref{E:Wrep2}. 
 \end{proof}
 
 \subsection{Reduction to the case when $X$ is bounded}
In this section, we only assume that $X\in L^2(\PP)$ with $\EE(X)=0$. We shall complete the proof of  Theorem \ref{T2} by a truncation argument, for which some notation is needed.

 For every  $b>0$, we set
$$X^{(b)}= \mathbf 1_{|X|\leq b}X - \EE(X \mathbf 1_{|X|\leq b}),$$ so $X^{(b)}$ is a centered and bounded variable; we write $\sigma^{(b)}$ for its standard deviation. 
Similarly, we set 
$$\hat X^{(b)}_n= \mathbf 1_{|\hat X_n|\leq b}\hat X_n - \EE(X \mathbf 1_{|X|\leq b})\quad \text{and} \quad \hat S^{(b)}(n)= \hat X^{(b)}_1+\cdots + \hat X^{(b)}_n.
$$
Clearly, $\hat S^{(b)}$ is a version of the step-reinforced random walk with typical step distributed as $X^{(b)}$. 
The latter being bounded, an application of Theorem \ref{T2} as proven in the preceding section shows that 
there is the convergence in distribution in the sense of Skorohod  
\begin{equation} \label{E:fin-3}n^{-1/2} \hat S^{(b)} (\lfloor \cdot n\rfloor)  \, \Longrightarrow \ \sigma^{(b)} \hat B(\cdot) \qquad \text{ as $n\to \infty$, }
\end{equation}
where $\hat B$ denotes a noise reinforced Brownian motion with reinforcement parameter $p$. 

Recall that the topology of  weak convergence on the set of probability measures on  the Skorokhod space of c\`adl\`ag paths is metrizable. Since plainly $\lim_{b\to\infty} \sigma^{(b)}=\sigma$, it follows readily from \eqref{E:fin-3} that we have also
\begin{equation} \label{E:fin-2}
n^{-1/2} \hat S^{(b(n))} (\lfloor \cdot n\rfloor)  \, \Longrightarrow \ \sigma \hat B(\cdot) \qquad \text{ as $n\to \infty$}
\end{equation}
for some sequence $(b(n))$ of positive real number that tends to $\infty$ slowly enough. 

Then consider
 $$\check S^{(b(n))}(n)=  \hat S(n)-\hat S^{(b(n))}(n),$$
 and observe that
 $$ \check S^{(b(n))}(n)= \check X^{(b(n))}_1+\cdots + \check X^{(b(n))}_n$$
 with
 $$\check X^{(b(n))}_n= 
 \mathbf 1_{|\hat X_n|> b(n)}\hat X_n - \EE(X \mathbf 1_{|X|> b(n)}).$$
In turn, $\check S^{(b(n))}$
 is also a step-reinforced random walk, now with typical step distributed as $X-X^{(b(n))}$.
The latter is a centered variable in $L^2(\PP)$, and if we write $\varsigma^{(b(n))}$ for its standard deviation, 
then clearly  $\lim_{n\to\infty} \varsigma^{(b(n))}=0$ since $b(n)$ tends to $\infty$. We deduce from Corollary \ref{C1} that for any $t>0$,
\begin{equation} \label{E:fin-1}\lim_{n\to \infty}n^{-1}\EE \left (\max_{k\leq nt}|\check S^{(b(n))}(k)|^2 \right) =0.
\end{equation}

We now see from \eqref{E:fin-1} and the Markov inequality that the  requirement  3.30 on page 316 in \cite{JS} holds for $Z^n_{\cdot}= n^{-1/2} \check S^{(b(n))}(\lfloor \cdot n\rfloor)$. 
This enables us to apply Lemma 3.31 there with $Y^n_{\cdot}= n^{-1/2} \hat S^{(b(n))} (\lfloor \cdot n\rfloor) $, and we conclude that 
$$n^{-1/2} \hat S (\lfloor \cdot n\rfloor) =  Y^n_{\cdot}+Z^n_{\cdot}  \, \Longrightarrow \ \sigma \hat B(\cdot) \qquad \text{ as $n\to \infty$.}$$
The proof of Theorem \ref{T2} is now complete. 

 \bibliography{NoiseR.bib}

\end{document}